\documentclass[twoside,11pt]{amsart}

\usepackage[english]{babel}
\usepackage[utf8x]{inputenc}
\usepackage{url}
\usepackage{amsmath}
\usepackage{graphicx}
\usepackage[colorinlistoftodos]{todonotes}
\usepackage{amsmath,amscd,amsthm}
\usepackage{amstext, amsfonts, a4}
\usepackage{amssymb}
\usepackage{tikz-cd}
\usepackage{fancyhdr}
\usepackage{enumerate}
\usepackage{cleveref}
\usepackage{xcolor}

\numberwithin{equation}{section}
\newtheorem{thm}[equation]{Theorem}

\newtheorem{prop}[equation]{Proposition}
\newtheorem{cor}[equation]{Corollary}

\theoremstyle{definition}

\newtheorem{ex}[equation]{Example}

\renewcommand{\dim}{\operatorname{\mathsf{dim}}}
\renewcommand{\deg}{\operatorname{\mathsf{deg}}}

\newcommand\ind{\operatorname{\mathsf{ind}}}

\newcommand\Nrd{\operatorname{\mathsf{Nrd}}}

\newcommand\N{\operatorname{N}}

\newcommand{\can}{\operatorname{\mathsf{can}}}

\newcommand{\vf}{\varphi}
\newcommand{\mg}[1]{{#1}^{\times}}
\newcommand{\sq}[1]{{#1}^{\times 2}}
\newcommand{\scg}[1]{\mg{#1}/\sq{#1}}
\newcommand{\s}{\sigma}
\newcommand{\nat}{\mathbb{N}}

\newcommand{\la}{\langle}
\newcommand{\ra}{\rangle}
\newcommand{\lla}{\la\!\la}
\newcommand{\rra}{\ra\!\ra}

\renewcommand{\leq}{\leqslant}

\newcommand{\I}{\mathsf{I}}
\newcommand{\rr}{\mathbb{R}}

\newcommand{\mc}{\mathcal}
\renewcommand{\N}{\mathsf{N}}

\newcommand{\disc}{\mathsf{disc}}
\newcommand{\Sim}{{\bf\mathsf{Sim}}}
\newcommand{\PSim}{{\bf\mathsf{PSim}}}
\newcommand{\G}{\mathsf{G}}
\renewcommand{\H}{\mathsf{H}}

\newcommand{\Z}{\mathsf{Z}}
\newcommand{\Hyp}{\mathsf{Hyp}}

\newcommand{\ovl}{\overline}
\renewcommand{\setminus}{\smallsetminus}

\renewcommand{\deg}{\mathsf{deg}}
\renewcommand{\dim}{\mathsf{dim}}
\renewcommand{\leq}{\leqslant}

\newcommand{\D}{\mathsf{D}}
\newcommand{\Cl}{\mc{C}}

\newcommand{\zz}{\mathbb{Z}}

\renewcommand{\setminus}{\smallsetminus}

\makeatletter
\newcommand{\bigperp}{%
  \mathop{\mathpalette\bigp@rp\relax}%
  \displaylimits
}

\newcommand{\bigp@rp}[2]{%
  \vcenter{
    \m@th\hbox{\scalebox{\ifx#1\displaystyle2.1\else1.5\fi}{$#1\perp$}}
  }%
}
\makeatother

\title[Non-$R$-trivial proper projective similitudes]{Non-$R$-trivial proper projective similitudes\\ in type $A_3\equiv D_3$}

    \date{9 May, 2026}

\author{M.~Archita}
\author{Karim Johannes Becher}

\address{University of Antwerp, Department of Mathematics, Antwerp, Belgium.}

\email{archita.mondal@uantwerpen.be}
\email{karimjohannes.becher@uantwerpen.be}
\begin{document}

 \maketitle

\begin{abstract}
Over an arbitrary field  of characteristic different from $2$ admitting an anisotropic torsion $3$-fold Pfister form, we apply a construction due to Merkurjev to produce an algebra with orthogonal involution of degree $6$ which admits proper projective similitudes that are not $R$-trivial. 
In particular, such examples exist over every finitely generated transcendental extension of a local or global number field, as well as over every finitely generated extension of transcendence degree $3$ of $\rr$.

\medskip\noindent
{\sc Keywords:} 
Classical adjoint linear algebraic group, stably rational, $R$-equivalence, quadratic form, algebra with involution, quaternion algebra, hyperbolic, Pfister form, cohomological dimension

\medskip\noindent
{\sc Classification (MSC 2020):} 11E04, 
11E57, 
11E81, 
14E08, 
20G15 
\end{abstract}

\section{Introduction}

Let $K$ be a field of characteristic different from $2$.
In \cite{Mer96}, A.~Merkurjev studied $R$-equivalence for classical groups $G$ of adjoint type over $K$.
Presenting $G$ as the group of proper projective similitudes of a $K$-algebra with involution $(A,\s)$, he described the group of  $R$-equivalence classes $G(K')/R$ for field extensions $K'/K$ in terms of a quotient of two subgroups of $\mg{K'}$ related to  factors of similitude of $\s$ and norms from finite extensions of $K'$ where the involution becomes hyperbolic. (See \Cref{Mer:T1} below.)

Merkurjev's work had been followed up in various directions.
Firstly, using Merkurjev's criterion, $R$-triviality of $G$ was established for certain types of groups; see e.g. \cite{PTW12} and \cite{AB25}.
Secondly, many examples of non-$R$-trivial groups $G$ were constructed; see e.g. \cite{Gil97}, \cite{BMT04}, \cite{Bha14} and \cite{AB25}. The wealth of examples suggest that $R$-triviality might be rather the exception than the default.
Thirdly, sufficient conditions on the base field $K$ were considered  for having triviality of $G(K)/R$ for adjoint classical groups of given types; see e.g.~\cite{KP08}, \cite{PS15}, \cite{PS17}, \cite{AP22} and \cite{AB26}. 

In \cite[Theorem~3]{Mer96}, $R$-triviality is completely characterized in terms of invariants for adjoint groups of type $A_3\equiv D_3$. 
In this article we exhibit sufficient conditions on the field $K$ for the presence of adjoint groups $G$ of type $A_3\equiv D_3$ over $K$ such that $G(K)/R$ is nontrivial.
This is stronger than having that $G$ is not $R$-trivial, since for the latter, the existence of a (possibly transcendental) extension $K'/K$ with $G(K')/R\neq \{1\}$ suffices.
We apply a construction in \cite{Mer96} to obtain examples of such groups $G$ of type $A_3\equiv D_3$ with $G(K)/R\neq\{1\}$ over any field $K$ that admits an anisotropic torsion $3$-fold Pfister form.

\section{Algebras with involution and similitudes}

By a \emph{$K$-algebra with involution} we mean a pair $(A,\s)$ where
$A$ is a finite-dimensional $K$-algebra endowed with an $K$-linear involution $\s$ such that $A$ has no nontrivial two-sided ideals $I$ with $\s(I)=I$ and, denoting by $\Z(A)$ is the center of $A$, we have $K=\{x\in\Z(A)\mid \s(x)=x\}$.
It follows that $\dim_KA=[\Z(A):K]\cdot n^2$ for some integer $n\in\nat$, denoted by $\deg A$ and called the \emph{degree of $A$}.
Algebras with involution are distinguished into three types: \emph{orthogonal}, \emph{symplectic} and \emph{unitary}; we refer to \cite[\S 2]{KMRT98} for details.

Let $(A, \sigma)$ be a $K$-algebra with an 
involution. An element $a\in \mg{A}$ is said to be a 
{\it similitude} if $\sigma(a)a \in \mg{K}$. 
The similitudes of $(A, \sigma)$ form a subgroup of $\mg{A}$, which is denoted by 
$\Sim(A, \sigma)$. We refer to 
\cite[\S 12]{KMRT98} for the basics on similitudes. 
We obtain a group homomorphism 
$$\mu: \Sim(A, \sigma) \to \mg{K},\quad  a\mapsto \sigma(a)a\,,$$
called the \emph{multiplier map}.
The image of this map is denoted by $\G(A, \sigma)$. 
Hence $$\G(A,\s)=\{x\in\mg{K}\mid x=\s(a)a\mbox{ for some }a\in\mg{A}\}\,.$$
We set $\PSim(A,\s)=\Sim(A,\s)/\mg{K}$ and call this the group of \emph{projective similitudes of $(A,\s)$}.
Note that $\PSim(A,\s)$ is naturally isomorphic to
the group of automorphisms of $(A,\s)$, in view of the Skolem-Noether Theorem.

If $\deg A=2m$ with $m\in\nat$ and $\s$ is orthogonal then, for $a\in\Sim(A,\s)$, we have 
$\Nrd_A(a)=\pm a^m$, and we set 
$$\Sim^+(A,\s) = \{a\in\Sim(A,\s)\mid \Nrd_A(a)=a^m\}\,.$$
If $\deg A$ is odd or $\s$ is not orthogonal, then we set $\Sim^+(A,\s)=\Sim(A,\s)$.
Hence $\Sim^+(A,\s)$ is always a subgroup of $\Sim(A,\s)$ of index $1$ or $2$.
We obtain corresponding subgroups 
$\PSim^+(A,\s)=\Sim^+(A,\s)/\mg{K}$ of $\Sim(A,\s)$ and $\G^+(A,\s)=\mu(\Sim^+(A,\s))$ of $\G(A,\s)$.

For a field extension $K'/K$, we obtain from $(A,\s)$ by scalar extension an $K'$-algebra with involution $(A_{K'},\s_{K'})$.
One obtains naturally a group scheme ${\bf PSim}^+(A,\s)$ such that, for any field extension $K'/K$, the set of $K'$-rational points 
${\bf PSim}^+(A,\s)(K')$ is given by $\PSim^+(A_{K'},\s_{K'})$.
Note that ${\bf PSim}^+(A,\s)$ is an adjoint, semi-simple, connected linear algebraic group over $K$.

Assume that $(A,\s)$ is a $K$-algebra with orthogonal involution of even degree.
We denote by $\Cl(A,\s)$ the \emph{Clifford algebra} of $(A,\s)$; see \cite[\S 8]{KMRT98} for the definition.
The center of $\Cl(A,\s)$ is a quadratic \'etale extension of $K$, hence of the form $K[X]/(X^2-d)$ for some $d\in\mg{K}$, and the square-class $d\sq{K}\in\scg{K}$ is uniquely determined by $\s$, denoted by $\disc(\s)$ and called the \emph{discriminant of $\s$}.

Consider two $K$-algebras with involution $(A_1, \s_1)$ and $(A_2,\s_2)$. 
We call $(A,\s)$ \emph{an orthogonal sum of  $(A_1, \s_1)$ and $(A_2,\s_2)$} if there exist $e\in A$ with $\s(e)=e$ and $e^2=e$ and $K$-algebra isomorphisms, $$\vf_1:A_1\xrightarrow{\sim}eAe\quad\mbox{ and }\quad \vf_2:A_2\xrightarrow{\sim}(1-e)A(1-e)$$
such that $ \vf_i\circ \s_i=\s\circ \vf_i$ for $i=1,2$. 
Note that this implies that $A\sim A_1\sim A_2$ and that $\s_1$ and $\s_2$ are of the same type as $\s$.
If $(A,\s)$ is an orthogonal sum of $(A_1,\s_1)$ and $(A_2,\s_2)$, we also indicate this by writing
$$(A,\s)\in (A_1,\s_1)\boxplus (A_2,\s_2)\,.$$

\section{$R$-triviality in groups of type $A_3\equiv D_3$}

The study of $R$-equivalence for semisimple adjoint classical groups 
can be reduced to such ones that are absolutely simple.
As a consequence of A.~Weil's classification results in \cite{Wei61}, any absolutely simple adjoint classical group over $K$ is 
given by $\mathbf{PSim}^+(A,\s)$ for some $K$-algebra with involution $(A,\s)$.
If $A$ is of even degree $2n$ and $\s$ is an orthogonal involution, then this group is of Dynkin type $D_n$.

\medskip

For a finite field extension $L/K$, we write 
$\N_{L/K}:L\to K$ for the norm map, and we abbreviate $\N_{L/K}^\ast=\N_{L/K}(\mg{L})$.
Similarly, for a central simple $K$-algebra $A$, we write 
$\Nrd_A:A\to K$ for the reduced norm map, and we abbreviate $\Nrd_A^\ast=\Nrd_A(\mg{A})$.

Let $(A,\s)$ be a $K$-algebra with orthogonal or symplectic involution.
We denote by $\Hyp(A,\s)$ the subgroup of $\mg{K}$ generated by the subsets $\N^\ast_{L/K}$ where $L/K$ ranges over the finite field extensions for which $\s_{L}$ is hyperbolic.
We set $$\H(A,\s)=\sq{K}\cdot\Hyp(A,\s)\,.$$
Merkurjev studies the 
group of $R$-equivalence classes of  
${\bf PSim}^+ (A, \sigma)$ by means of the following translation.

\begin{thm}[Merkurjev \hbox{\cite[Theorem 1]{Mer96}}]
\label{Mer:T1}
We have $\H(A,\sigma)\subseteq\G^+(A,\sigma)$ and
$$ 
{\bf PSim}^+(A,\sigma)/R \simeq  
\G^+(A,\sigma)/\H(A,\sigma).  
$$ 
In particular, ${\bf PSim}^+(\s)$ is $R$-trivial if and only if $\G^+(A_{K'},\s_{K'})=\H(A_{K'},\s_{K'})$ holds for every field extension $K'/K$.
\end{thm}

This provides tool to compute $R$-equivalence on semisimple adjoint classical groups and to construct examples of nonrational adjoint groups. 
The simplest examples occur for algebras with orthogonal involutions of degree $6$.
In the split case, where $(A,\s)$ is adjoint to a regular $6$-dimensional quadratic form $\eta$ over $K$, \cite[Theorem 2]{Mer96} yields that ${\bf PSim}^+(A,\s)$ is $R$-trivial if and only if $\mc{C}(A,\s)$ (the even Clifford algebra of $\eta$) is not a division algebra.
Using this one obtains examples over fields of cohomological dimension $3$ where $\mathbf{PSim}^+(A,\s)$ is not $R$-trivial.
However, there seems to be no such example known where $\mathbf{PSim}^+(A,\s)(K)\neq \{1\}$ over a field of cohomological dimension $3$ with $A$ split.
We now turn our attention to the nonsplit case.

\begin{prop}\label{D3-ex-construction}
Let $(A,\s)$ be a $K$-algebra with orthogonal involution of degree $6$ and of nontrivial discriminant. Let $C=\Cl(A,\s)$ and $L=\Z(C)$. Assume that $\ind(C)=2$ and $A$ is not split. 
Then there exist $K$-quaternion algebras $Q,Q_1,Q_2$ and an orthogonal involution $\gamma$ on $Q$ such that $L$ is contained in $Q$ and $\gamma$ extends the nontrivial automorphism of $L/K$ and such that $A\sim Q\sim Q_1\otimes Q_2$ and $$(A,\s) \in (Q_1,\can_{Q_1})\otimes (Q_2,\can_{Q_2}) \boxplus (Q,\gamma)\,.$$
\end{prop}
\begin{proof}
    This is explained in \cite[p.~207-208]{Mer96}.
\end{proof}

The following statement corresponds to \cite[Prop.~9]{Mer96}, where it is stated in a context, not highlighting the minimal hypotheses. 
We include a proof.

\begin{thm}[Merkurjev] 
\label{T:MerkProp9}
    Let $Q,Q_1,Q_2$ be $K$-quaternion  algebras such that $Q_1\otimes Q_2\sim Q$.
    Let $L/K$ be a quadratic field extension contained in $Q$ and $\gamma$ an orthogonal involution on $Q$ extending the nontrivial automorphism of $L/K$.
    Let $(A,\s)\in (Q_1,\can_{Q_1})\otimes (Q_2,\can_{Q_2})\boxplus (Q,\gamma)$ and $C=\Cl(A,\s)$.
    Then $C$ is a central simple $L$-algebra
with $C\sim (Q_1)_L\simeq (Q_2)_L$, and we have 
$$\G^+(A,\s)  = \N_{L/K}^\ast\cap(\Nrd_{Q_1}^\ast\cdot \Nrd_{Q_2}^\ast) \quad\mbox{ and }\quad
\H(A,\s) = \N^\ast_{L/K}\cap \Nrd^\ast_{Q_k}$$
for $k=1,2$, and in partiuclar
    $${\bf PSim}^+(A,\s)(K)/R\,\,\simeq \,\,\dfrac{\N_{L/K}^\ast\cap(\Nrd_{Q_1}^\ast\cdot \Nrd_{Q_2}^\ast)}{\N^\ast_{L/K}\cap \Nrd^\ast_{Q_k}}\,\,.$$
\end{thm}

\begin{proof}
Note that $\disc(\s)=\disc(\gamma)$, by \cite[Def.~7.2 \& Prop.~7.5]{KMRT98}.
    We fix $d\in\mg{K}$ such that $\disc(\gamma)=\disc(\s)=d\sq{K}$. 
    Then $L\simeq  K(\sqrt{d})$.
    
  Set $(B,\tau)=(Q_1,\can_{Q_1})\otimes (Q_2,\can_{Q_2})$.
    Then $\disc(\tau)$ is trivial, and therefore $\G^+(A,\s)\subseteq \N_{L/K}^\ast=\G^+(Q,\gamma)$.
    We conclude that 
    $\G^+(A,\s)=\N_{L/K}^\ast\cap \G^+(B,\tau)$.
    Since $B\sim Q$, we have that $\ind B\leq 2$.
    Hence
    \cite[Prop.~4.2~$(c)$]{AB26}
     yields that $\G^+(B,\tau)=\Nrd^\ast_{Q_1}\cdot \Nrd^\ast_{Q_2}$.
    This shows that $$\G^+(A,\s)=\N^\ast_{L/K}\cap \Nrd^\ast_{Q_1}\cdot \Nrd^\ast_{Q_2}\,.$$
    
Consider an arbitrary finite field extension $M/K$ such that $\sigma_{M}$ is hyperbolic. 
Then $\disc(\s_{M})$ is trivial and $C_{M}$ is split. 
Hence we may view $L$ as a subfield of $M$, and obtain by \cite[Prop.~4.1~$(b)$]{AB26} that $\N_{M/K}^\ast=\N_{L/K}(\N_{M/L}^\ast)\subseteq\N_{L/K}(\Nrd_C^\ast)$.
Having this for every finite field extension $M/K$ such that $\s_{M}$ is hyperbolic, we conclude that $\Hyp(\sigma)\subseteq\N_{L/K}(\Nrd_C^\ast)$.

By \cite[p.~93, Examples.~$(b)$]{KMRT98} we have $\Cl(B,\tau)=Q_1\times Q_2$.
Note that $Q_L$ is split and $\gamma_L$ is hyperbolic.
This yields that $C\sim (Q_1)_L\sim (Q_2)_L$.

Fix $k\in\{1,2\}$.
Then $\Nrd_C^\ast=\Nrd_{(Q_k)_L}^\ast$, by \cite[Prop.~4.1~$(a)$]{AB26}.
In particular $\Hyp(A,\sigma)\subseteq\N_{L/K}(\Nrd_C^\ast)=\N_{L/K}(\Nrd_{(Q_k)_L}^\ast)$.
    By \cite[Lemma 10]{Mer96}, we have 
    $$\H(A,\s)=\sq{K}\Hyp(\s)\subseteq \sq{K}\N_{L/K}(\Nrd_{(Q_k)_L}^\ast)= \N_{L/K}^\ast\cap \Nrd_{Q_k}^\ast.$$
    To show the converse inclusion, consider now $a\in\N_{L/K}^\ast\cap\Nrd_{Q_k}^\ast$.
    If $(Q_k)_L$ is split, then 
    $\s_L$ is hyperbolic and hence $a\in\Hyp(A,\s)$.
    Assume that $(Q_k)_L$ is not split.
    Since $a\in\Nrd_{Q_k}^\ast$, there exist a quadratic extension $L'/K$ such that 
    $(Q_k)_{L'}$ is split and $a\in\N_{L'/K}^\ast$.
    Then $L'/K$ is linearly disjoint from $L/K$, so we obtain a biquadratic field extension $M=L\otimes_KL'$ of $K$.
    By \cite[Lemma 2.11]{PTW12}, we have  $\N_{L/K}^\ast\cap \N_{L'/K}^\ast=\sq{K}\cdot \N_{M/K}^\ast$.
    Hence $a\in\sq{K}\cdot \N_{M/K}^\ast$.
    Since $\s_{M}$ is hyperbolic, we conclude that $a\in\sq{K}\cdot\Hyp_2(A,\s)\subseteq \H(A,\s)$.  
    This shows that $$\H(A,\s)=\N_{L/K}^\ast\cap \Nrd_{Q_k}^\ast\quad\text{ for }\quad k\in\{1,2\}\,.$$
The final part of the statement now follows by \Cref{Mer:T1}.
\end{proof}

\begin{prop}\label{P:3-Pfister-A3-example}
    Let $a,b,c\in\mg{K}$ be such that $-1\in\D_K{\lla a,b\rra}$ and
    $\lla a,b,c\rra$ is anisotropic.
Let $Q_1=(a,b)_K$, $Q_2=(-ac,b)_K$, $Q=(-c,b)_K$ and $L=K(\sqrt{-c})$. 
 Then $$c\in\left(\N_{L/K}^\ast\cap(\Nrd_{Q_1}^\ast\cdot \Nrd_{Q_2}^\ast)\right)\setminus (\Nrd_{Q_1}^\ast\cup\Nrd_{Q_2}^\ast)$$
and there exists an orthogonal involution $\gamma$ on $Q$ with $\disc(\gamma)=-c\sq{K}$.
Furthermore, there exists a $K$-algebra of degree $6$ with orthogonal involution $$(A,\s)\in \left((Q_1,\can_{Q_1})\otimes (Q_2,\can_{Q_2})\right)\boxplus (Q,\gamma)\,,$$ and for any such $(A,\s)$ we have ${\bf PSim}^+(A,\s)(K)/R\neq \{1\}$.
\end{prop}
\begin{proof}
Note that $c=\N_{L/K}(\sqrt{-c})\in\N_{L/K}^\ast$, $-1,-a\in \Nrd_{Q_1}^\ast$ and $ac\in \Nrd_{Q_2}^\ast$. 
Therefore $c\in\N_{L/K}^\ast\cap (\Nrd_{Q_1}^\ast\cdot\Nrd_{Q_2}^\ast)$.
On the other hand, since $\lla a,b,c \rra$ is anisotropic, we have that $c\notin\Nrd_{Q_1}^\ast\cup\Nrd_{Q_2}^\ast$.
The statement now follows by \Cref{D3-ex-construction} and \Cref{T:MerkProp9}.
\end{proof}

For a valuation $v$ on $K$, we denote by $vK$ its value group and by $Kv$ its residue field.

\begin{cor}
    Let $a,b,t\in \mg{K}$ and assume that there exists a valuation $v$ on $K$ with $v(2)=v(a)=v(b)=0$,  and $v(t)\notin 2vK$.
    Assume that $a$ is a sum of two squares in $K$ and that the $Kv$-quaternion algebra
    $(\ovl{\vphantom{b}a},\ovl{b})_{Kv}$ is not split.
    Set $Q_1=(a,b)_{K}$, $Q_2=(-at,b)_K$ and $L=K(\sqrt{-t})$.
    Then $$t\in \N^\ast_{L/K}\cap \left(\N_{L/K}^\ast\cap(\Nrd_{Q_1}^\ast\cdot \Nrd_{Q_2}^\ast)\right)\setminus  (\Nrd^\ast_{Q_1}\cup\Nrd^\ast_{Q_2})\,.$$ 
\end{cor}

\begin{proof}
Since $v(t)\notin 2 v K$ and $(\ovl{\vphantom{b}a},\ovl{b})_{Kv}$ is not split, we have that $t\notin\Nrd_{Q_1}^\ast$, so the quadratic form $\lla a,b,t\rra$ over $K$ is anisotropic.
Since $a$ is a sum of two squares, we have $-1\in\D_K\lla a,b\rra$.
We conclude by \Cref{P:3-Pfister-A3-example}, taking $c=t$.
\end{proof}

As a special case we retrieve \cite[Example 6.1]{PS17}. It partially inspired this note.

\begin{ex}
Let $K=\mathbb{Q}_p(t)$ for an odd prime number $p$.
Let $u\in \zz$ be a non-square modulo $p$.
Then $u$ is a sum of two squares in $K$ and 
$\lla p, u, t\rra$ is anisotropic over $K$ and equal to $\lla p,u,-pt\rra$.
Let $Q_1=(u,p)_K$, $Q_2=(u,t)_K$, $Q=(u,pt)_K$, $L=K(\sqrt{pt})$ and let $\gamma$ be an orthogonal involution on $Q$ with $\disc(\gamma)=pt\sq{K}$.
Hence, with \Cref{P:3-Pfister-A3-example}, we find an algebra with orthogonal involution $(A,\s)$ of degree $6$ such that $-pt\in\G(A,\s)\setminus\H(A,\s)$, whereby ${\bf PSim}^+(A,\s)(K)\neq \{1\}$.
\end{ex}

The same construction can be applied to obtain similar examples over any field admitting an anisotropic torsion $3$-fold Pfister form.

\begin{thm}\label{T:main}
Assume that there exists an anisotropic torsion $3$-fold Pfister form over $K$.
Then there exists a $K$-algebra with orthogonal involution $(A,\s)$ of degree $6$ such that ${\bf PSim}^+(A,\s)(K)/R\neq\{1\}$.    
\end{thm}
\begin{proof}
    By \cite[Chap.~XI, Theorem 4.5]{Lam05}, the hypothesis implies that there exist $a,b,c\in\mg{K}$ such that $a$ is a sum of two squares in $K$ and the $3$-fold Pfister form $\lla a,b,c\rra$ over $K$ is anisotropic.
We set 
$Q_1=(a,b)_K$, $Q_2=(-ac,b)_K$, $Q=(-c,b)_K$ and $L=K(\sqrt{-c})$.
It follows by \Cref{P:3-Pfister-A3-example}  that
$$\N_{L/K}^\ast\cap(\Nrd_{Q_1}^\ast\cdot \Nrd_{Q_2}^\ast)\neq \N^\ast_{L/K}\cap \Nrd^\ast_{Q_1}\,.$$
Since $L$ is a subfield of $Q$, the nontrivial automorphism of $L/K$ extends to an  orthogonal involution $\gamma$ on $Q$. 
We take a $K$-algebra with orthogonal involution $$(A,\s)\in (Q_1,\can_{Q_1})\otimes (Q_2,\can_{Q_2})\boxplus (Q,\gamma)\,.$$
Then it follows by \Cref{T:MerkProp9} that ${\bf PSim}^+(A,\s)(K)/R\neq \{1\}$.
\end{proof}

In the case where $K$ is nonreal, we can now characterize when exactly all proper projective similitudes are $R$-trivial.

\begin{cor}
    Assume that $K$ is nonreal. Then ${\bf PSim}^+(A,\s)(K)/R=\{1\}$ for every central simple $K$-algebra with orthogonal involution $(A,\s)$ of degree $6$ if and only if $\I^{3}K=0$.    
\end{cor}

\begin{proof}
    Note that $\I^3K$ is torsion, because $K$ is nonreal.
    If $\I^3K\neq 0$, then by \Cref{T:main} there exists a central simple $F$-algebra with orthogonal involution $(A,\s)$ of degree $6$ such that ${\bf PSim}^+(A,\s)(K)/R\neq\{1\}$.   
    If $\I^{3}K=0$, then by \cite[Prop.~3.2]{KP08}, every central simple $K$-algebra with orthogonal involution $(A,\s)$ satisfies $\G(A,\s)=\H(A,\s)$ and consequently ${\bf PSim}^+(A,\s)(K)/R=\{1\}$.    
\end{proof}

\subsection*{Acknowledgments}
This work was supported by the \emph{Bijzonder Onderzoeksfonds, University of Antwerp} (project \emph{BOF Opvang MSCA IF}, ID 51418).

\section*{Declarations}

\subsection*{Data availability}
There is no associated data, not  contained in the article.
\\[-6mm]

\subsection*{Conflict of interest} The authors declare that there is no conflict of interest.

\bibliographystyle{plain}

\begin{thebibliography}{HH}
\bibitem{AB25} 
M. Archita, K.J. Becher. 
{Rational connectedness for groups of proper projective similitudes}. 
\emph{Preprint (2025)},
\url{https://doi.org/10.48550/arXiv.2506.21717}


\bibitem{AB26} 
M. Archita, K.J. Becher. 
{Similitudes over fields with $I^4=0$}. 
\emph{Preprint (2026)},\\
\url{https://doi.org/10.48550/arXiv.2602.22147}

\bibitem{AP22} M. Archita, R. Preeti.
Rational equivalence on adjoint groups of type $D_{n}$ over fields of virtual cohomological dimension $2$.
\emph{Trans. Amer. Math. Soc.} 375 (2022), 7373--7384.



\bibitem{Bha14} N.~Bhaskhar. 
More examples of non-rational adjoint groups. 
\emph{J.~Algebra} 397 (2014), 39--46. 

 

\bibitem{BMT04} G.~Berhuy, M.~Monsurrò, J.-P.~Tignol. 
Cohomological invariants and $R$-triviality of adjoint classical groups. \emph{Math.~Z.} 248 (2004), 313--323.

\bibitem{Gil97}
P. Gille. 
Examples of non-rational varieties of adjoint groups. \emph{J. Algebra} 193 (1997), 728--747.

\bibitem{KMRT98} M. Knus, A. Merkurjev, M. Rost, J.-P. Tignol. \emph{The Book of Involutions}.
Amer. Math. Soc. Colloq. Publ., 44
American Mathematical Society, Providence, RI, 1998.

\bibitem{KP08}  A. Kulshrestha, R. Parimala. \emph{$R$-equivalence in adjoint classical groups over fields of virtual cohomological dimension 2}. \emph{Trans. Amer. Math. Soc.} 360 (2008), 1193--1221.  
 
\bibitem{Lam05} T.Y.~Lam. \newblock \emph{Introduction to quadratic forms over fields}.
Grad. Stud. Math., 67
American Mathematical Society, Providence, RI, 2005.
	

 
\bibitem{Mer96} A.~S. Merkurjev.
$R$-equivalence and rationality problem for semisimple adjoint classical algebraic groups. \emph{Publ.~Math.~de l'IH\'ES} 84.1 (1996), 189--213. 

 

\bibitem{PS15}  R.~Preeti, A.~Soman.
Adjoint groups over $\mathbb{Q}_p(X)$ and $R$-equivalence.
\emph{J. Pure Appl. Algebra} 219 (2015), 4254--4264.

\bibitem{PS17}  R.~Preeti, A.~Soman. 
Adjoint groups over $\mathbb{Q}_p(X)$ and R-equivalence -- revisited. \emph{Proc. Amer. Math. Soc.} 145.3 (2017), 1019--1029.

\bibitem{PTW12} R.~Parimala, J.-P.~Tignol, R.M.~Weiss. The Kneser-Tits conjecture for groups with Tits-index over an arbitrary field. \emph{Transformation Groups} 17.1 (2012), 209--231.

       
		
\bibitem{Wei61} A. Weil. 
Algebras with involutions and the classical groups. \emph{J. Ind. Math. Soc.} 24 (1961), 589--623.

\end{thebibliography}

\end{document}